\documentclass[psamsfonts,a4paper]{amsart}

\usepackage{amssymb,amsfonts}
\usepackage{amsaddr}
\usepackage[all,arc]{xy}
\usepackage{enumerate}
\usepackage{tikz}
\usetikzlibrary{matrix}
\usepackage{parskip}
\usepackage[margin=2.5cm]{geometry}
\usepackage{eucal}

\newtheorem{thm}{Theorem}[section]

\newtheorem{prop}[thm]{Proposition}
\newtheorem{lemma}[thm]{Lemma}
\newtheorem{conj}[thm]{Conjecture}

\theoremstyle{definition}
\newtheorem{defn}[thm]{Definition}

\theoremstyle{remark}

\makeatletter
\let\c@equation\c@thm
\makeatother
\numberwithin{equation}{section}

\DeclareMathOperator{\aut}{Aut}

\DeclareMathOperator{\Hom}{Hom}

\DeclareMathOperator{\tr}{Tr}

\newcommand{\C}{\mathbb{C}}

\newcommand{\R}{\mathbb{R}}
\newcommand{\Q}{\mathbb{Q}}
\newcommand{\Z}{\mathbb{Z}}

\newcommand{\CC}{\mathcal{C}}
\newcommand{\HH}{\mathcal{H}}

\newcommand{\Y}{\mathcal{Y}}
\newcommand{\CF}{\mathcal{F}}

\newcommand{\g}{\mathfrak{g}}
\newcommand{\h}{\mathfrak{h}}

\newcommand{\al}{\alpha}
\newcommand{\ga}{\gamma}
\newcommand{\G}{\Gamma}
\newcommand{\D}{\Delta}

\newcommand{\la}{\lambda}
\newcommand{\La}{\Lambda}
\newcommand{\om}{\omega}

\newcommand{\Om}{\Omega}

\newcommand{\ov}{\overline}

\newcommand{\vir}{\text{Vir}}

\newcommand{\twobytwo}[4]
{\left(\begin{smallmatrix} #1 & #2 \\ #3 & #4 \end{smallmatrix}\right)}

\newcommand{\bigtwobytwo}[4]
{\left(\begin{matrix} #1 & #2 \\ #3 & #4 \end{matrix}\right)}

\DeclareMathOperator{\sign}{sign}

\DeclareMathOperator{\id}{id}
\DeclareMathOperator{\ev}{ev}
\DeclareMathOperator{\coev}{coev}
\DeclareMathOperator{\qdim}{qdim}

\title[Lattices, Vertex Algebras, and Modular Categories]{Lattices, Vertex Algebras, and Modular Categories}
\author[*]{Jethro van Ekeren}

\begin{document}

\begin{center}
{\LARGE \bf Lattices, Vertex Algebras and Modular Categories} \par \bigskip

\renewcommand*{\thefootnote}{\fnsymbol{footnote}}
{\normalsize
Jethro van Ekeren\footnote{email: \texttt{jethrovanekeren@gmail.com}}
}

\par \bigskip

{\footnotesize Instituto de Matem\'{a}tica e Estat\'{i}stica \\ Universidade Federal Fluminense (UFF) \\ Niter\'{o}i RJ, Brazil}

\par \bigskip
\end{center}

\vspace*{8mm}

\noindent
\textbf{Abstract.} 
In this note we give an account of recent progress on the construction of holomorphic vertex algebras as cyclic orbifolds as well as related topics in lattices and modular categories. We present a novel computation of the Schur indicator of a lattice involution orbifold using finite Heisenberg groups and discriminant forms.
\vspace*{8mm}

\section{Introduction}

Vertex algebras formalise the calculus of operator product expansions in two dimensional conformal field theory (CFT) \cite{BPZ} \cite{B86}. Various structures of CFT, such as fusion rules, may be analysed in terms of the representation theory of the corresponding vertex algebra \cite{FZ92} \cite{HLtensorIandII}. In particular a vertex algebra $V$ is said to be holomorphic if the set of fusion rules of the corresponding CFT is as simple as possible, i.e., if all $V$-modules are simply direct sums of copies of $V$ itself.

The problem of classifying holomorphic vertex algebras resembles the problem of classifying even self dual lattices. In both cases strong constraints arise from the theory of modular forms. For example, the central charge of a holomorphic vertex algebra must be divisible by $8$, just as the rank of an even self dual lattice must be divisible by $8$.

Let $L$ be an even lattice of rank $r$. Then there is a vertex algebra $V_L$ of central charge $c=r$ canonically associated with $L$, and the irreducible representations of this lattice vertex algebra are indexed by the discriminant form $L^\vee / L$. In particular if $L$ is self dual then $V_L$ is holomorphic. In physical terms $V_L$ corresponds to the theory of strings in the torus $\R^r / L$.

It turns out that all holomorphic vertex algebras of central charges $c=8$ and $c=16$ are lattice vertex algebras \cite{DMaff04}, but that for central charge $c=24$ the situation is different. Indeed there are precisely $24$ self dual even lattices, but there are more than $24$ holomorphic vertex algebras. The most prominent among these non lattice holomorphic vertex algebras is the moonshine module $V^\natural$, whose symmetry group is the sporadic finite simple group known as the monster \cite{FLM}.

The original construction of $V^\natural$ from the lattice vertex algebra $V_\La$ associated with the Leech lattice $\La$ is an example of a cyclic orbifold (more specifically a $\Z/2$-orbifold). This construction mirrors a classical procedure for constructing new lattices from old ones known as Kneser's neighbour construction \cite{Kneser}. It is known that any even self dual lattice can be obtained from any other of the same rank by repeated application of the neighbour construction.

Conjecturally a similar state of affairs holds for vertex algebras (at least in the case of central charge $24$). Indeed the known holomorphic vertex algebras of central charge $24$ are $71$ in number and the list of such is conjectured to be complete \cite{S93}. Since all these vertex algebras are constructed as iterated ($\Z/2$- and more generally $\Z/n$-) orbifolds of lattice vertex algebras, and since the cyclic orbifold construction is invertible (see Section \ref{section.ab.grp.coho}), it would follow that any of these vertex algebras may be obtained from any other by repeated application of the cyclic orbifold construction.

In this note we give an account of the theory of $\Z/n$-orbifolds of holomorphic vertex algebras developed by S. M\"{o}ller, N. Scheithauer, and the author \cite{EMS15} \cite{EMS17}, along with background material on lattices, vertex algebras and tensor categories. We discuss in particular a subtle invariant of self-dual modules over a group or algebra known as the Schur indicator. In a sense the vertex algebra Schur indicator measures the Fermi/Bose statistics of an orbifold. We compute some Schur indicators in terms of classical objects such as finite Heisenberg groups and discriminant forms.

\newpage

\emph{Acknowledgements.} I would like to thank the organisers of the conference ``String Geometries and Dualities'' IMPA 2016, as well as my collaborators S. M\"{o}ller and N. Scheithauer. I would also like to thank T. Creutzig and S. Kanade for useful discussions, and the referee for comments on the manuscript.

\section{Vertex Algebras}

The rational function $1/(z-w)$ may be expanded as a Laurent series in positive powers of $w$ whose coefficients are negative powers of $z$, or vice versa. These two expansions are naturally viewed as elements of the rings $\C((z))((w))$ and $\C((w))((z))$ respectively.

A vertex algebra is an algebraic structure in which differently ordered and differently parenthesised products are related to each other in roughly the same way that the two expansions of $1/(z-w)$ above are. Formally a vertex algebra is a vector space $V$ together with infinitely many products
\[
a \otimes b \mapsto a_{(n)}b \quad \text{for $n \in \Z$}
\]
such that $a_{(n)}b = 0$ for any fixed $a, b \in V$ and sufficiently large and positive $n \in \Z$ depending on $a$ and $b$. We gather these products into a generating function as follows
\[
  Y : V \otimes V \rightarrow V((z)), \quad a \otimes b \mapsto Y(a,z)b = \sum_{n \in \Z} a_{(n)}b z^{-n-1},
\]
the operator valued series $Y(a, z)$ is known as the field associated with $a \in V$. The products are to satisfy the condition that for all $a, b, c \in V$ the three terms
\[
Y(a, z)Y(b, w)c, \quad Y(b, w)Y(a, z)c, \quad \text{and} \quad Y(Y(a, z-w)b, w)c
\]
be images of a unique element of $V[[z, w]][z^{-1}, w^{-1}, (z-w)^{-1}]$ under the following natural maps
\begin{align*}
\xymatrix{
& & V[[z, w]][z^{-1}, w^{-1}, (z-w)^{-1}] \ar@{->}[dll] \ar@{->}[d] \ar@{->}[drr] & & \\
V((z))((w)) & & V((w))((z)) & & V((z-w))((w)). \\
}
\end{align*}
We depict this axiom schematically as
\begin{align}\label{assoc.rel}
Y(a, z)Y(b, w)c \sim Y(b, w)Y(a, z)c \sim Y(Y(a, z-w)b, w)c.
\end{align}
Additionally the vertex operation $Y$ is assumed to have a unit element $\mathbf{1}$ which satisfies $Y(\mathbf{1}, z) = 1_V$ and $\left.Y(a,z)\mathbf{1}\right|_{z=0} = a$, and a translation operator $T : V \rightarrow V$ which satisfies $[T, Y(a,z)] = \partial_z Y(a, z)$ and $T\mathbf{1}=0$ \cite{KacVA}.

There are deep links between vertex algebras and algebraic curves, but for these notions to make contact one needs to be able to interpret the formal variable $z$ as a coordinate \cite{zhu1994global} \cite{FBZ}, and to do this one needs to be able to say how $Y(a, z)$ varies depending on changes of the coordinate. This dependence is encoded in terms of a conformal structure \cite{Huang.2dCFT.Book}. A \emph{conformal structure} on a vertex algebra $V$ is an element $\om \in V$ whose associated field
\[
  Y(\om, z) = L(z) = \sum L_n z^{-n-2}
\]
furnishes $V$ with a representation of the Virasoro Lie algebra $\vir = \bigoplus_{n\in \Z} \C L_n \oplus \C C$, namely
\[
[L_m, L_n] = (m-n)L_{m+n} + \delta_{m,-n} \frac{m^3-m}{12}C, \quad \text{and} \quad [C, \vir]=0.
\]
The representation is to be ``compatible'' in the sense that
\begin{itemize}
\item $T = L_{-1}$,

\item The action of $C$ is by a constant $c$ called the central charge of $V$,

\item The vertex algebra is $\Z_+$-graded by eigenvalues of $L_0$:
\[
V = V_0 \oplus V_1 \oplus V_2 \oplus \cdots.
\]
\end{itemize}
In this paper all vertex algebras are assumed to be conformal.

There are several variations on these definitions. For instance the $\Z_+$-grading might be relaxed to a $\Z$-grading, a $\tfrac{1}{2}\Z$-grading, or a more general grading. Also important is the notion of vertex superalgebra $V = V^0 \oplus V^1$, in which the axiom (\ref{assoc.rel}) is modified so that elements of the odd part $V^1$ anticommute rather than commute.

It remains to mention that there are several technical conditions on a vertex algebra which we shall need to impose at one point or another to ensure the validity of our statements. For the sake of flow we avoid presenting many theorems under their most general hypotheses, while to maintain precision we define once and for all a \emph{tame} vertex algebra to be one that is rational (see Section \ref{module.def}), $C_2$-cofinite, simple, self-contragredient, and of CFT-type. All these conditions are standard in the vertex algebra literature.

\section{Modules and Twisted Modules}\label{module.def}

Just as for a Lie algebra, the notions of module and intertwining operator between modules make sense for vertex algebras \cite{FHL}. Thus a module over the vertex algebra $V$ consists of a vector space $M$ together with a map $Y^M : V \otimes M \rightarrow M((z))$ satisfying a relation
\[
Y^M(Y(a, z-w)b, w) m \sim Y^M(a, z)Y^M(b, w)m \sim Y^M(b, w)Y^M(a, z) m
\]
analogous to (\ref{assoc.rel}). In particular $Y$ equips $V$ itself with the structure of a $V$-module. If every $V$-module is decomposable into a direct sum of irreducible $V$-modules then we say that $V$ is rational.

An intertwining operator between $V$-modules $M, N, R$ is a map
\begin{align}\label{int.def}
\Y : M \otimes N \rightarrow z^{\varepsilon} \cdot R((z)),
\end{align}
compatible with the three $V$-actions as follows
\[
\Y(Y^M(a, z-w)m, w)n \sim Y^R(a, z)\Y(m, w)n \sim \Y(m, w)Y^N(a, z)n.
\]
The reader might wonder why the factor $z^\varepsilon$ has been inserted in (\ref{int.def}). It turns out that an irreducible module $M$ over a tame vertex algebra $V$ possesses a grading
\begin{align}\label{module.grading}
M = \bigoplus_{n \in \Z_+} M_{\rho + n},
\end{align}
by finite dimensional eigenspaces of the Virasoro operator $L^M_0$. The number $\rho = \rho(M)$ is called the \emph{conformal weight} of $M$. By {\cite[Theorem 11.3]{DLM00}} the conformal weight is necessarily rational. We abuse terminology by also refering to the important $\Q/\Z$-valued invariant
\[
\D(M) = \rho(M) \pmod{1}
\]
as the conformal weight of $M$. Compatibility of Virasoro algebra actions dictates that nontrivial intertwining operators can only exist if
\[
\varepsilon \equiv \D(M) + \D(N) - \D(R).
\]
The vector space of intertwining operators $\Y$ as in (\ref{int.def}) is denoted $\binom{R}{M \,\, N}$.

It is actually useful to incorporate fractional exponents, such as that which appears in (\ref{int.def}), directly into the definition of $V$-module. More precisely if $V$ comes equipped with a $\Z/n$-grading $V = \bigoplus_{k \in \Z/n} V^k$, then one may demand that for $a \in V^k$, $m \in M$ the series $Y^M(a, z) m$ take the form $z^{k/n} \cdot M((z))$.

After choosing a primitive $n^{\text{th}}$ root of unity, a $\Z/n$-grading becomes the same thing as an automorphism of order $n$. Thus it is equivalent to speak of \emph{$g$-twisted} $V$-modules, where $g$ is a finite order automorphism of $V$. The subspace $V^g \subset V$ of $g$-invariant vectors is a vertex subalgebra, and by restriction any $g$-twisted $V$-module becomes a $V^g$-module in the ordinary sense. Conversely if $V^\la$ denotes in general the $g$-eigenspace with eigenvalue $e^{2\pi i \la}$ then the restriction
\[
V^\la \otimes V^\mu \rightarrow V^{\la + \mu},
\]
of the vertex operation $Y$ to $V^\la \otimes V^\mu$ yields a simple example of intertwining operator among $V^g$-modules.

A technical condition concerning the conformal weights of $V$-modules will be required in the sequel. We call the tame vertex algebra $V$ \emph{positive} if the conformal weight $\rho(M)$ of every irreducible $V$-module $M$ (with the exception of $V$ itself) is positive.

\section{Integral Lattices and Discriminant Forms}

An integral lattice is a free abelian group $L$ of finite rank, equipped with a symmetric bilinear form $(\cdot,\cdot) : L \times L \rightarrow \Z$. Unless otherwise stated, we assume all lattices to be positive definite and nondegenerate. The lattice $L$ is even if the norm $\al^2=(\al,\al)$ of $\al$ is an even number for every $\al \in L$.

The dual
\[
  L^\vee = \{x \in \Q \otimes_\Z L | \text{$(x, \al) \in \Z$ for all $\al \in \Z$}\},
\]
obviously contains $L$, and the quotient group $D = L^\vee / L$ is finite. If $L$ is even then $D$ carries an induced {quadratic form} $q : D \rightarrow \Q / \Z$, defined by
\[
q(\al) = \al^2/2 \pmod{1}.
\]
The pair $(D, q)$ is called the discriminant form of $L$, where in general a \emph{discriminant form} is by definition a finite abelian group $E$ together with a function $q : E \rightarrow \Q / \Z$ such that $q(n\al) = n^2q(\al)$ for all $n \in \Z$, and such that $(\cdot,\cdot) : E \times E \rightarrow \Q$ defined by
\[
  (\al, \beta) = q(\al+\beta)-q(\al)-q(\beta)
\]
is a nondegenerate, bilinear form.

\begin{defn}
A subgroup $I$ of a discriminant form $(E, q)$ is called \emph{isotropic} if $(\al,\beta) = 0$ for all $\al, \beta \in I$. It is called \emph{totally isotropic} if $q(\al)=0$ for all $\al \in I$.
\end{defn}

\section{Lattice Vertex Algebras}

The simplest examples of tame vertex algebras are constructed from even lattices \cite{FK80} \cite{B86}.

Let $L$ be an even integral lattice, and denote by $\h$ the vector space $\C \otimes_\Z L$ equipped with the induced bilinear form $(\cdot,\cdot)$. The Heisenberg algebra associated with $\h$ is the infinite dimensional Lie algebra
\[
\widehat{\h} = \h((t)) \oplus \C \mathbf{1}, \quad \text{where} \quad [ht^m, ht^n] = \delta_{m, -n} (h, h') \mathbf{1}, \quad \text{and} \quad [\mathbf{1}, \widehat{\h}] = 0.
\]
The Fock module $\CF_\mu$ is the highest weight $\widehat{\h}$-module generated by the highest weight vector $\left|\mu\right>$ on which $t\h[t]$ acts trivially and $h \in \h$ acts as the scalar $(\mu, h)$. As a vector space the lattice vertex algebra associated with $L$ is the direct sum of Fock modules
\[
V_L = \bigoplus_{\al \in L} \CF_\al.
\]
For the construction of the vertex operation $Y$ of $V_L$ we refer the reader to the textbook \cite{KacVA}. The central charge of $V_L$ is just the rank of $L$.

Let $\mu \in L^\vee/L$. Then a slight modification of the construction of $V_L$ can be used to furnish the direct sum
\[
V^{\mu} = \bigoplus_{\al \in \mu+L} \CF_\al
\]
with the structure of a $V_L$-module. This $V_L$-module is irreducible, and in fact this construction yields a bijection from the discriminant form $L^\vee/L$ to the set of irreducible $V_L$-modules.

Now let $g$ be an automorphism of $L$, and consider the assignment $\left|\al\right> \mapsto \eta(\al) \left|g \cdot \al\right>$ for some system of coefficients $\eta : L \rightarrow \C^\times$. If $\eta$ satisfies a certain cocycle condition, coming from compatibility with the vertex operation, then the assignment extends uniquely to an automorphism of $V_L$. If the restriction of $\eta$ to the fixed point sublattice $L^g$ is trivial then we call this automorphism a \emph{standard lift} of $g$. It turns out that any two standard lifts of $g \in \aut{L}$ are conjugate in $\aut(V_L)$.

\section{Even Self Dual Lattices}

We say that a lattice $L$ is self dual if $L^\vee = L$. The conditions of self duality, evenness, and positive definiteness, although easily satisfied in isolation, combine to put remarkable constraints on a lattice. In particular the rank of such a lattice must be divisible by $8$. The smallest nontrivial example is
\[
E_8 = \left\{ (x_1,\ldots,x_8) \in \Z^8 \cup (\tfrac{1}{2}+\Z)^8 \Big| {\textstyle\sum_i} x_i \in 2\Z \right\},
\]
which is in fact the unique even self dual lattice in rank $8$.

In rank $16$ there are two even self dual lattices: the orthogonal sum $E_8 \oplus E_8$, as well as another known as $D_{16}^+$. In rank $24$ there are $24$ even self dual lattices, known collectively as the Niemeier lattices. In rank $32$ there are at least $10^7$, and beyond this point the number grows very rapidly with rank.

Let $L$ be an even self dual lattice, let $b \in L$, and put
\[
L_b = \left\{ x \in L | (b, x) \in 2\Z \right\}.
\]
Up to equivalence there are two possibilities for the discriminant form $(D, q)$ of $L_b$. The two possibilities are represented graphically in the following figure.
\begin{center}
\begin{tikzpicture}

\tikzset{
quangleprod/.style={
    matrix of nodes,
    column sep=-\pgflinewidth, row sep=-\pgflinewidth,
    nodes={draw,
      minimum height=#1,
      anchor=center,
      text width=#1,
      align=center, 
      inner sep=0pt 
    },
  },
  quangleprod/.default=1cm
}

\matrix[quangleprod]
{
$0$ & $1/4$ \\
$0$ & $3/4$ \\
};

\matrix[quangleprod] at (6, 0)
{
$0$ & $0$ \\
$0$ & $1/2$ \\
};
\end{tikzpicture}
\end{center}
In either case $D \cong \Z/2 \times \Z/2$, the top left cell is the identity element $[L_b] \in D$, the addition in $D$ is the obvious one, and the number in each cell represents the value of $q$. The two possibilities arise, respectively, in the cases $b^2 \equiv 2 \pmod{4}$ and $b^2 \equiv 0 \pmod{4}$.

If for some nonzero $a \in D$ we have $q(a) = 0$, then the union of cosets
\[
L_b \cup (a+L_b)
\]
is again an even self dual lattice. In the first case above the union of cosets corresponding to the left column simply recovers the original lattice $L$. In the second case we obtain $L$ again from the left column (say), but we obtain a \emph{new} even self dual lattice, which we denote by $L^\circ$, from the top row.

The procedure $L \mapsto L^\circ$ is Kneser's neighbour construction \cite{Kneser} \cite{Bor.thesis}. Depending on the choice of $b$ the new lattice $L^\circ$ might or might not be isomorphic to $L$. Kneser showed that any even self dual lattice of given rank can be obtained from any other by iterated application of the construction. Since in rank $8k$ we already have the even self dual lattice $E_8^{\oplus k}$ as a starting point, we can construct all even self dual lattices in principle.

\section{The Fusion Product}\label{sec.fusion}

Let $L_1$ and $L_2$ be two integrable modules of level $k \in \Z_+$ over the affine Kac-Moody algebra
\[
  \widehat{\g} = \g((t)) \oplus \C K.
\]
The tensor product $L_1 \otimes L_2$ is again a $\widehat{\g}$-module, but now of level $2k$. It emerged from the work of physicists on conformal field theory that there is another tensor product between integrable $\widehat{\g}$-modules, this one preserving the level \cite{KZ}. This tensor product
\[
L_1, L_2 \mapsto  L_1 \boxtimes L_2
\]
was shown by Kazhdan and Lusztig to be equivalent in a certain sense to the tensor product of modules over a quantum group at a root of unity \cite{KL1}.

There is a similar tensor product between representations of a tame vertex algebra \cite{MS89} \cite{HLtensorIandII}. Let us consider the (complexified) Grothendieck group $\CF(V)$ of the category of representations of the tame vertex algebra $V$. The tensor product induces a commutative ring structure $*$ on $\CF(V)$, which we call the \emph{fusion algebra} of $V$. For irreducible $V$-modules $A$ and $B$ one has
\begin{align}\label{ten.fus}
[A] * [B] = \sum_{C} \mathcal{N}_{A, B}^C \cdot [C]
\end{align}
where the sum runs over isomorphism classes of irreducible $V$-modules $C$. The structure constants $\mathcal{N}_{A, B}^C \in \Z_+$ are known as the fusion rules of $V$. The fusion rules are related to intertwining operators by the relation
\begin{align}\label{fus.int}
\mathcal{N}_{A, B}^C = \dim\binom{C}{A \,\,\, B}.
\end{align}

Let us see what this structure looks like in the case of the lattice vertex algebra $V_L$. We have noted that the irreducible $V_L$-modules are parametrised by the discriminant form $D = L^\vee/L$. In fact the group structure of $D$ mirrors the fusion product of $\CF(V_L)$, i.e.,
\begin{align*}
V^\la \boxtimes V^\mu \cong V^{\la+\mu}.
\end{align*}
Furthermore the conformal weight of $V^{\mu}$ is given by
\[
  \D(V^{\mu}) = q(\mu).
\]
So the category of representations of $V_L$ can be viewed as a sort of categorification of the discriminant form $(D, q)$ of $L$.

\section{Modular Forms}\label{Section.Zhu}

Let $V$ be a tame vertex algebra and $M$ an irreducible $V$-module graded as in (\ref{module.grading}). We consider its graded dimension
\[
\chi_M = q^{-c/24} \sum_{n=0}^\infty q^{\rho+n} \dim(M_{\rho+n}).
\]
Zhu showed in \cite{Z96} that the series $\chi_M$ converges to a holomorphic function of $\tau \in \HH$ where $q = e^{2\pi i \tau}$ and $\HH$ is the upper half complex plane. Furthermore the action of the modular group $SL_2(\Z)$ on $\HH$ by fractional linear transformations preserves the span of the graded dimensions of the irreducible $V$-modules. To be more precise the vector space $\CF(V)$ carries a representation $\rho_V$ of $SL_2(\Z)$, the operator $\rho_V(A)$ for $A \in SL_2(\Z)$ can be written as a matrix $\rho_{M, N}\left(A\right)$ relative to the basis of $\CF(V)$ consisting of classes of irreducible $V$-modules, and the characters of $V$ satisfy
\[
\chi_M\left(A \cdot \tau\right) = \sum_{N} \rho_{M, N}\left(A\right) \chi_N.
\]

Now let $L$ be an even lattice of rank $r$ and discriminant form $D$. The theta function $\Theta_\mu$ of $\mu \in D$ and the Dedekind eta function $\eta$ are defined by
\[
\Theta_\mu = \sum_{\al \in \mu + L} q^{\al^2/2} \quad \text{and} \quad \eta = q^{1/24} \prod_{n=1}^\infty \left( 1-q^n \right).
\]
The character of the irreducible $V_L$-module $V^\mu$ is $\chi_{V^\mu} = \Theta_\mu / \eta^r$. Zhu's representation $\rho_V$ of $SL_2(\Z)$ on $\CF(V_L) \cong \C[D]$ is determined by the action of generators $T = \twobytwo 1101$ and $S = \twobytwo{0}{-1}{1}{0}$ as follows.
\begin{align}\label{Zhu.representation.lattice}
\begin{split}
T [\al] &= e^{2\pi i \left( q(\al) - r/24 \right)} [\al] \\
S [\al] &= \frac{1}{\sqrt{\# D}} \sum_{[\beta] \in D} e^{-2\pi i (\al,\beta)} [\beta].
\end{split}
\end{align}
The representation $\rho_V$ is closely related to the \emph{Weil representation} $\rho_D$. Indeed $\rho_V$ is the tensor product of $\rho_D$ with a character \cite{Borch.grass}.

If the lattice $L$ is self dual then the lattice vertex algebra $V_L$ has only a single irreducible module, namely itself. In general if a tame vertex algebra enjoys this property then we say it is \emph{holomorphic}.

Now we consider the graded dimension of a holomorphic vertex algebra $V$. Zhu's result states that $\chi_V$ is a modular function with possible character. If the central charge is $24$ it easily follows that the only possibility is for $\chi_V$ to coincide with the classical $j$ function up to an additive constant, i.e.,
\[
\chi_V(q) = q^{-1} + \dim(V_1) + 196884 q + 21493760 q^2 + \cdots.
\]
Thus modular invariance puts stringent conditions on a holomorphic vertex algebra.

\section{Fusion Algebras of Fixed Point Vertex Algebras}\label{fps.fusion}

In Section \ref{sec.fusion} we saw that for the lattice vertex algebra $V_L$ the irreducible modules are parametrised by the abelian group $D = L^\vee/L$, and that the fusion product is given by addition in $D$.

In general let us say that the tame vertex algebra $V$ has \emph{grouplike fusion} if it enjoys a similar property, that is, if there is a parametrisation
\[
\al \mapsto W^\al
\]
of the set of irreducible $V$-modules by an abelian group $E$ such that
\[
W^\al \boxtimes W^\beta \cong W^{\al+\beta}
\]
for all $\al, \beta \in E$.

Now let $V$ be a holomorphic vertex algebra and $g$ an automorphism of $V$ of order $n$. It has been proved by Miyamoto and collaborators that the fixed point subalgebra $V^g$ is tame and possesses $n^2$ irreducible modules \cite{Miyamoto.Tanabe} \cite{Miyamoto.C2} \cite{Miyamoto.Carnahan}.

A theorem of Dong-Li-Mason \cite{DLM00} asserts that there exists a unique irreducible $g$-twisted $V$-module, which we denote by $V(g)$. Theorem 5.11 of {\cite{EMS15}} asserts that the grading (\ref{module.grading}) of $V(g)$ takes the form
\begin{align}\label{V(g).grading}
V(g) = \bigoplus_{j\in \Z_+} V(g)_{r/n^2 + j/n}
\end{align}
for some integer $r$ defined modulo $n$.
\begin{defn}
  If $V(g)$ carries the grading (\ref{V(g).grading}) then we say that the automorphism $g$ has \emph{type} $n\{r\}$.  
\end{defn}
The first main theorem of \cite{EMS15} asserts that the fixed point algebra $V^g$ has grouplike fusion (see also \cite{DRX}), and that the structure of its fusion algebra $\CF(V^g)$ depends only on the type of $g$.

Before stating the theorem we recall that group extensions
\begin{align}\label{Zn.extension}
0 \rightarrow \Z/n \rightarrow E \rightarrow \Z/n \rightarrow 0
\end{align}
are classified up to equivalence by the invariant $d \in \Z/n$ defined as follows. If $i$ is a generator of $\Z/n$ and $\widehat{i}$ its lift from the right hand copy of $\Z/n$ to $E$, then $\widehat{i}^n$ lies in the kernel of the projection, and hence is the image of $d \cdot i$ for some $d \in \Z$ defined modulo $n$. In other words $H^2(\Z/n, \Z/n) \cong \Z/n$.
\begin{thm}[\cite{EMS15}]\label{thm.fus.grp}
Let $V$ be a tame holomorphic vertex algebra, and $g$ an automorphism of $V$ of type $n\{r\}$. Let $E$ be the extension \textup{(}\ref{Zn.extension}\textup{)} with invariant $d=2r$. Then there is a parametrisation
\[
\al \mapsto W^\al
\]
of the set of irreducible $V$-modules by $E$ such that for all $\al, \beta \in E$ one has
\[
W^\al \boxtimes W^\beta \cong W^{\al+\beta}.
\]
\end{thm}
We remark that each twisted $V$-module $V(g^i)$ is naturally a $V^g$-module and as such splits into a direct sum of $n$ irreducible modules, and in this way the $n^2$ irreducible $V^g$-modules are obtained from the $n$ twisted $V$-modules. In Theorem \ref{thm.fus.grp} the subgroup $\Z/n \subset E$ parametrises those $V^g$-modules which occur inside the untwisted module $V$ itself. The projection $E \rightarrow \Z/n$ indexes within which twisted $V$-module the $V^g$-module occurs as a submodule.

The following theorem gives a more complete description of the structure of $\CF(V^g)$ in terms of the type of $g$ by including the conformal weights of the irreducible modules.
\begin{thm}
Let $V$ be a tame holomorphic vertex algebra, and $g$ an automorphism of $V$ of type $n\{r\}$. Let $(D, q)$ be the discriminant form of the indefinite lattice $L$ defined by
\[
L = \Z \oplus \Z \quad \text{with Gram matrix} \quad \bigtwobytwo{-2r}{n}{n}{0}.
\]
Then there is a bijection from $D$ to the set of irreducible $V^g$-modules, which we denote $\al \mapsto W^\al$, such that for all $\al, \beta \in D$ one has
\[
W^\al \boxtimes W^\beta \cong W^{\al+\beta} \quad \text{and} \quad \D(W^\al) = q(\al).
\]
\end{thm}

\section{The Orbifold Construction and the Monster}

The lattice vertex algebra $V_L$ associated with a self dual even lattice $L$ is holomorphic.

In fact the only holomorphic tame vertex algebras of central charge $8$ and $16$ are the lattice vertex algebras associated with $E_8$, $E_8 \oplus E_8$ and $D^+_{16}$. The Niemeier lattices give rise to $24$ holomorphic vertex algebras of central charge $24$, but there are others that do not arise this way. The most famous of these is the moonshine module $V^\natural$, a sketch of whose construction by Frenkel-Lepowsky-Meurman \cite{FLM} we now describe.

Among the Niemeier lattices the most remarkable is the Leech lattice $\La$, distinguished by its lack of vectors of norm $2$. The symmetry group of $\La$ is isomorphic to $2.\text{Co}_1$ where the Conway group $\text{Co}_1$ is one of the sporadic finite simple groups.

Now consider the involution of $\La$ which sends any vector $\al$ to its negative $-\al$, and denote by $\sigma$ a standard lift of this involution to the lattice vertex algebra $V_\La$. It turns out that $\sigma$ has type $2\{0\}$. According to the general picture described in Section \ref{fps.fusion}, the invariant subalgebra $V_\La^\sigma$ has four irreducible modules, whose fusion group structure and conformal weights are as shown in the following diagram.
\begin{center}
\begin{tikzpicture}

\tikzset{
quangleprod/.style={
    matrix of nodes,
    column sep=-\pgflinewidth, row sep=-\pgflinewidth,
    nodes={draw,
      minimum height=#1,
      anchor=center,
      text width=#1,
      align=center, 
      inner sep=0pt 
    },
  },
  quangleprod/.default=1.5cm
}

\matrix[quangleprod]
{
$V_\La^\sigma$ & $V(\sigma)^+$ \\
$V_\La^-$ & $V(\sigma)^-$ \\
};

\matrix[quangleprod] at (6, 0)
{
$0$ & $0$ \\
 $0$ & $1/2$ \\
};
\end{tikzpicture}
\end{center}
We observe that the discriminant form $D = \Z/2 \times \Z/2$ with quadratic form $\D$, contains two totally isotropic subgroups isomorphic to $\Z/2$.

As a vector space the moonshine module is the direct sum
\[
V^\natural = V_\La^\sigma \oplus V(\sigma)^+.
\]
By (\ref{ten.fus}) and (\ref{fus.int}) there exists a unique up to scalar intertwining operator
\[
\Y : V(\sigma)^+ \otimes V(\sigma)^+ \rightarrow V_\La^\sigma((z)).
\]
The vertex operation on $V_\La^\sigma$, the $V_\La^\sigma$-module structure of $V(\sigma)^+$, and a judicious choice of the intertwining operator $\Y$, together define a vertex algebra structure on $V^\natural$. See the next section.

Associated with the other isotropic subgroup we have the direct sum $V_\La^\sigma \oplus V_\La^-$. This sum obviously carries a vertex algebra structure as well, since it is just $V_\La$.





\section{Abelian Group Cohomology}\label{section.ab.grp.coho}

Suppose we are given a tame vertex algebra $W$ with grouplike fusion, i.e., a parametrisation
\[
\al \mapsto W^\al
\]
of the set of irreducible $V$-modules by a finite abelian group $E$, such that for all $\al, \beta \in E$,
\[
W^\al \boxtimes W^\beta \cong W^{\al+\beta}.
\]
In particular we may choose, for each $\al, \beta \in E$, a nonzero intertwining operator
\[
\Y^{\al,\beta}(-,z) : W^\al \otimes W^\beta \rightarrow z^{\varepsilon} \cdot W^{\al+\beta}((z))
\]
where
\[
\varepsilon = \D(W^\al) + \D(W^\beta) - \D(W^{\al+\beta}).
\]

The two compositions
\begin{align*}
\Y^{\al+\beta,\ga}(\Y^{\al, \beta}(a, z-w)b, w) c \quad \text{and} \quad \Y^{\al, \beta+\ga}(a, z) \Y^{\beta,\ga}(b, w) c
\end{align*}
are nonzero and proportional {\cite[Lemma 4.1]{Huang00GenRat}} {\cite[Theorem 3.9]{Huang05DEIO}}. The constant of proportionality, which depends only on $\al, \beta$ and $\gamma$, thus defines a map
\[
F : E \times E \times E \rightarrow \C^\times.
\]
A similar comparison of $\Y^{\al, \beta}$ with $\Y^{\beta,\al}$ yields a map
\[
\Om : E \times E \rightarrow \C^\times.
\]

Now let us consider a certain category $\CC_E$ constructed from the abelian group $E$. The set of objects of $\CC_E$ is $E$ itself, and the set $\Hom_{\CC_E}(\al, \beta)$ is $\C^\times$ if $\al = \beta$ and empty otherwise. It turns out {\cite{H08rigidity}} {\cite[Section 3.2]{Sven.thesis}} that the datum $(F, \Om)$ satisfies certain coherence relations which permit it to be used to endow $\CC_E$ with the structure of braided tensor category. We briefly recall this notion.


A braided tensor category is a category $\CC$ equipped with an operation $\otimes : \CC \times \CC \rightarrow \CC$ and natural isomorphisms
\[
a_{A, B, C} : (A \otimes B) \otimes C \rightarrow A \otimes (B \otimes C) \quad \text{and} \quad c_{A, B} : A \otimes B \rightarrow B \otimes A
\]
satisfying coherence conditions
\begin{align*}
\xymatrix{
 & ((A \otimes B) \otimes C) \otimes D \ar@{->}[dl]^{a \otimes 1} \ar@{->}[dr]^{a} & \\
(A \otimes (B \otimes C)) \otimes D \ar@{->}[d]^{a} \ar@{->}[d]^{c} & & (A \otimes B) \otimes (C \otimes D) \ar@{->}[d]^{a} \\
A \otimes ((B \otimes C) \otimes D) \ar@{->}[rr]^{1 \otimes a} & & A \otimes (B \otimes (C \otimes D)) \\
}
\end{align*}
and
\begin{align*}
\xymatrix{
 & (A \otimes B) \otimes C \ar@{->}[dl]^{a} \ar@{->}[d]^{c \otimes 1} & A \otimes (B \otimes C) \ar@{->}[d]^{a^{-1}} \ar@{->}[dr]^{1 \otimes c} &  \\
A \otimes (B \otimes C) \ar@{->}[d]^{c} \ar@{->}[d]^{c} & (B \otimes A) \otimes C \ar@{->}[d]^{a} & (A \otimes B) \otimes C \ar@{->}[d]^{c} \ar@{->}[d]^{c} & A \otimes (C \otimes B) \ar@{->}[d]^{a^{-1}} \\
(B \otimes C) \otimes A \ar@{->}[dr]^{a} & B \otimes (A \otimes C) \ar@{->}[d]^{1 \otimes c} & C \otimes (A \otimes B) \ar@{->}[d]^{a^{-1}} & (A \otimes C) \otimes B \ar@{->}[dl]^{c \otimes 1} \\
 & B \otimes (C \otimes A) & (C \otimes A) \otimes B. &  \\
}
\end{align*}

We obtain a braided tensor category structure on $\CC_E$ by putting $\al \otimes \beta = \al+\beta$, $a_{\al,\beta,\ga} = F(\al,\beta,\ga)$ and $c_{\al,\beta} = \Om(\al,\beta)$. In concrete terms, the pair $(F, \Om)$ satisfies the coherence conditions that derive naturally from the three commutative diagrams above. In \cite{ML52} MacLane introduced a cohomology theory for abelian groups (distinct from the usual group cohomology, which is defined for all groups) in which a $3$-cocycle is exactly a pair $(F, \Om)$ satisfying these coherence relations. In this theory a $2$-cochain is an arbitrary map
\[
\varphi : E \times E \rightarrow \C^\times,
\]
and the corresponding $3$-coboundary $d\varphi$ is the pair $(F, \Om)$ defined by
\[
F(\al,\beta,\ga) = \frac{\varphi(\beta,\ga)\varphi(\al, \beta+\ga)}{\varphi(\al+\beta,\ga)\varphi(\al,\beta)} \quad \text{and} \quad \Om(\al,\beta) = \frac{\varphi(\al,\beta)}{\varphi(\beta,\al)}.
\]
Modification of a $3$-cocycle by a coboundary yields an equivalent braided tensor category structure on $\CC_E$. Thus we might denote by $\CC_E(\om)$ the braided tensor category associated with $\om \in H^3_{\text{ab}}(E, \C^\times)$.

In the context of intertwining operators with which we began, modification of $(F, \Om)$ by a coboundary corresponds to a redefinition of the $\Y^{\al,\beta}$ by nonzero scalars. It follows that a tame vertex algebra $V$ with grouplike fusion has a canonically associated cohomology class
\[
\om_V \in H^3_{\text{ab}}(E, \C^\times).
\]
In general a collection of $V$-modules and intertwining operators that commute up to scalars $F$ and $\Om$ as described above is known as an abelian intertwining algebra \cite{DL93}.

Now let us suppose that for some subgroup $I \subset E$ the restriction of the class $\om_V$ vanishes in the cohomology group $H^3_{\text{ab}}(I, \C^\times)$. This means that the intertwining operators $\Y^{\al,\beta}$ can be chosen so that the sum $\bigoplus_{\al \in I} W^\al$ acquires the structure of a vertex algebra with vertex operation defined so that its restriction to $W^{\al} \otimes W^\beta$ coincides with $\Y^{\al,\beta}$ \cite{HohnGenera} \cite{carnahan.2014.building}.

The trace $\tr(\om)$ of an element $\om = [(F, \Om)] \in H^3_{\text{ab}}(E, \C^\times)$ is by definition the quadratic form $q$ on $E$ given by
\[
\Om(\al, \al) = e^{2\pi i q(\al)}.
\]
A theorem of Eilenberg-MacLane \cite{EILML54} asserts that the assignment $\om \mapsto \tr(\om)$ defines a bijection from $H^3_{\text{ab}}(E, \C^\times)$ to the set of quadratic forms on $E$ with values in $\C^\times$.

Now we come to the main point of this discussion. We have two quadratic forms on $E$, namely the conformal weight $\D$ and the trace $\tr(\om_V)$. It is straightforward to prove that the bilinear forms associated with these two quadratic forms effectively coincide (more precisely one is the negative of the other). Comparing the quadratic forms themselves is much more difficult, however we have the following result.
\begin{prop}[\cite{EMS15}]\label{trace.conf}
  Let $V$ be a tame vertex algebra with grouplike fusion, and associated discriminant form $(E, q)$. If $V$ is positive then
  \[
    \tr(\om_V) = -\D.
  \]
\end{prop}
In particular if $I \subset E$ is a totally isotropic subgroup with respect to $\D$, then the direct sum of $V$-modules indexed by $I$ carries a unique vertex algebra structure extending that of $V$. Combining this result with the computation of the fusion algebras of fixed point vertex algebras in Section \ref{fps.fusion} yields the following theorem.
\begin{thm}\label{orb.theorem}
  Let $V$ be a tame holomorphic vertex algebra, and let $g$ be an automorphism of $V$ of type $n\{0\}$. Then the fusion group of $V^g$ takes the form $I \times I^\vee$ where $I$ and $I^\vee$ are totally isotropic subgroups, each isomorphic to $\Z/n$. If $V^g$ is positive, then the direct sums
  \begin{align*}
\bigoplus_{\al \in I}W^\al \quad \text{and} \quad \bigoplus_{\al \in I^\vee}W^\al
  \end{align*}
  are both holomorphic tame vertex algebras. One of them is just $V$, and the other we denote $V^{\text{orb}(g)}$.
\end{thm}
The vertex algebra $V^{\text{orb}(g)}$ is known as the orbifold of $V$ by $g$. We may think of the orbifold construction as an analogue of Kneser's neighbour construction for lattices.

By its construction the orbifold $V^{\text{orb}(g)}$ is graded by $I^\vee \cong \Z/n$, and can thereby be equipped with an automorphism of order $n$, canonical up to Galois symmetry. Denoting some choice of such an automorphism by $g^\vee$, we then have $(V^{\text{orb}(g)})^{\text{orb}(g^\vee)} \cong V$. In \cite{EMS15} this is called the inverse orbifold construction.

\section{New Holomorphic Vertex Algebras from Old}

Let us return to the question of constructing holomorphic vertex algebras of central charge $24$. The lattice vertex algebra construction produces $24$ examples, associated with the Niemeier lattices. Based on the remarkable work of Schellekens \cite{S93}, and the accumulation of subsequent evidence, the following conjecture has been made. See also \cite{Hoehn17}.
\begin{conj}\label{71.conj}
There are exactly $71$ holomorphic tame vertex algebras of central charge $24$.
\end{conj}

Let $L$ be a Niemeier lattice and let $g$ be the standard lift to $V_L$ of an automorphism of $L$ of order $n$. Since the structure of $V_L$ is well understood, the type of $g$ can be determined easily. Indeed the conformal weight of the twisted module $V(g)$ is given by the formula
\begin{align}\label{lattice.conformal.weight}
\rho = \frac{1}{4n^2} \sum_{j=1}^{n-1} j(n-j)\dim(\h[\zeta^j]),
\end{align}
where $\zeta$ is a primitive $n^{\text{th}}$ root of unity, and $\h[\mu]$ denotes the eigenspace in $\h = L \otimes_\Z \C$ with eigenvalue $\mu$.

By running through all Niemeier lattices, and all their automorphisms of type $0$, we produce a huge set of holomorphic vertex algebras as orbifolds. But how do we know if the algebras produced are distinct or not? For this we need an invariant.

It is well known that the component of degree one of a tame vertex algebra $V$ carries the natural structure of a reductive Lie algebra. If $V$ is holomorphic of central charge $24$ then modular invariance of the character of $V$ together with some Lie algebra invariant theory can be used to constrain $V_1$. The following theorem of Schellekens is the result upon which Conjecture \ref{71.conj} above is based.
\begin{thm}[\cite{S93}, \cite{EMS15}]
Let $V$ be a holomorphic tame vertex algebra of central charge $24$. Then the Lie algebra $V_1$ is isomorphic to one among $71$ reductive Lie algebras, listed in Table 1 of \cite{S93}. 
\end{thm}
Calculations with Hauptmoduln are used to prove the following theorem which relates the dimensions of the Lie algebras associated with $V$, $V^g$, and $V^{\text{orb}(g)}$.\newpage
\begin{thm}[\cite{EMS17}]\label{dimension.formula}
Let $n$ be a positive integer such that the congruence subgroup $\G_0(n) \subset SL_2(\Z)$ has genus zero. Let $V$ be a holomorphic tame vertex algebra of central charge $24$ and $g$ an automorphism of $V$ of type $n\{0\}$ such that the conformal weight $\rho$ of the twisted module $V(g^i)$ satisfies $\rho(V(g^i)) \geq 1$ for $i=1,2,\ldots,n-1$. Then
\[
\dim(V_1^{\text{orb}(g)}) = 24 + \sum_{d|n} c_d \dim(V_1^{\sigma^d}).
\]
Here the coefficients $c_d$ are given by
\[
c_d = \frac{\la(d)}{d} \cdot \frac{\phi((d, n/d))}{(d, n/d)} \cdot \psi(n/d)
\]
where $\phi(n) = \sum_{d|n}d$ is Euler's totient function,
\begin{align*}
\la(n) = \prod_{\substack{\text{$p$ prime} \\ p|n}}(-p)
\quad \text{and} \quad
\psi(n) = n \prod_{\substack{\text{$p$ prime} \\ p|n}}(1 + 1/p).
\end{align*}

\end{thm}
It is evident that the Lie algebra $V^{\text{orb}(g)}_1$ contains $V^g_1$ as a Lie subalgebra. Since $V^{\text{orb}(g)}$ is constructed as a direct sum over $\Z/n$, it is in fact the case that $V^{\text{orb}(g)}_1$ contains $V^g_1$ as the fixed point Lie subalgebra under some automorphism of order $n$. In \cite{Kac69} {\cite[Chapter 8]{KacIDLA}} Kac gives a neat classification of finite order automorphisms of semisimple Lie algebras in terms of labeled Dynkin diagams in a way that allows the fixed point subalgebra to be read off from the corresponding diagram.

For some choices of $L$ and $g$ the facts gathered above conspire to uniquely fix the isomorphism type of the weight one Lie algebra of the orbifold $V_L^{\text{orb}(g)}$. This is illustrated by the following example. Let $L$ be the Niemeier lattice whose root system is of type $A_9^2 D_6$. Then $V_L$ admits an automorphism $g$ of type $4\{0\}$ with fixed point Lie subalgebra $(V_L^g)_1$ isomorphic to $A_2 B_2 D_5 \C$. On the other hand $\dim(V_L^{\text{orb}(g)})_1 = 96$, so by inspection of Schellekens' list the isomorphism type of $(V_L^{\text{orb}(g)})_1$ must be either
\[
A_2^{12}, \quad B_2^4 D_4^2, \quad A_2^2 A_5^2 B_2, \quad A_2^2 A_8, \quad \text{or} \quad A_2 B_2 E_6.
\]
But Kac's classification implies that, among these types, only $E_6$ admits an automorphism of order $4$ whose fixed point subalgebra contains $D_5$. Thus by process of elimination the weight one Lie algebra of $V_L^{\text{orb}(g)}$ is of type $A_2 B_2 E_6$. This sort of argument can be implemented on a computer.

We now summarise the current state of Conjecture \ref{71.conj}. All $71$ Lie algebras on Schellekens' list have been realised as weight one Lie algebras of holomorphic vertex algebras of central charge $24$, as follows.
\begin{itemize}
\item The lattice vertex algebras associated with the 24 Niemeier lattices \cite{B86} \cite{FLM}.

\item 15 examples constructed as $\Z/2$-orbifolds of Niemeier lattice vertex algebras by the $(-1)$-isometry \cite{DGM96} (and \cite{FLM} for the case of the moonshine module $V^\natural$).

\item 17 examples constructed as iterated $\Z/2$-orbifolds of lattice vertex algebras \cite{Lam11} \cite{LS12a} \cite{LS12b}. The combinatorial structure of a \emph{Virasoro frame} \cite{DGHframed} is used in these cases to keep track of the vertex algebra structure throughout the process of repeated orbifolds.

\item 1 example constructed as a $\Z/3$-orbifold of a lattice vertex algebra in \cite{M10Z3orbifold} and a further 2 in \cite{SSZ3orbifold} by the same method. Many of the results of Section \ref{fps.fusion} were first obtained in the order $3$ case in \cite{M10Z3orbifold}.

\item 6 examples constructed as $\Z/2$-orbifolds of previously constructed vertex algebras (two of them conditionally on the next item) \cite{LS3} \cite{LL}.

\item 5 examples constructed as $\Z/n$-orbifolds of Niemeier lattice vertex algebras by lattice automorphisms (of orders $n = 4, 5, 6, 6, 10$) \cite{EMS15}.

\item 1 example constructed as a $\Z/7$-orbifold of the Leech lattice vertex algebra $V_\La$ \cite{LS-A67}.
\end{itemize}
The uniqueness of the holomorphic vertex algebra with given weight one Lie algebra has been established for most cases. At the time of writing $63$ out of $71$ cases have been settled, as follows.
\begin{itemize}
\item For the 24 Niemeier lattice vertex algebras uniqueness is proved in \cite{dong.mason.eff}.

\item For 2 of the framed vertex algebras uniqueness was shown in \cite{LS12b}.

\item For 1 vertex algebra uniqueness is derived in \cite{LL} from its construction as a \emph{mirror extension} \cite{DJXmirror}.
  
\item For 36 vertex algebras constructed as cyclic orbifolds of lattice vertex algebras uniqueness is deduced from the uniqueness of the lattice vertex algebra and the existence of an \emph{inverse \textup{(}or reverse\textup{)} orbifold} \cite{LS.reverse} \cite{KLL} \cite{EMS17} \cite{LS.Leech.uniqueness}.
\end{itemize}

\section{Modular Categories}\label{sec.mod.cat}

In Section \ref{section.ab.grp.coho} we obtained a braided tensor category $\CC_E(\om_V)$ from the tame vertex algebra $V$ with grouplike fusion and fusion group $E$. This structure is really a shadow of a more elaborate structure present in the category of representations of any tame vertex algebra, namely the structure of a modular category \cite{MS89} \cite{fuchs2002tft} \cite{H08rigidity}.

A pre-modular category is a $\C$-linear semisimple braided tensor category $\CC$, an assignment of a dual object $A^*$ to every object $A$ together with morphisms
\[
\coev_A : \mathbf{1} \rightarrow A \otimes A^* \quad \text{and} \quad \ev_A : A^* \otimes A \rightarrow \mathbf{1},
\]
as well as a system of {``twist''} isomorphisms $\theta_A : A \rightarrow A$ satisfying
\[
\theta_{A \otimes B} = (\theta_A \otimes \theta_B) \circ c_{B,A} \circ c_{A,B} \quad \text{and} \quad \theta_{A^*} = \theta_A^*.
\]
There exists a remarkable diagrammatic calculus of braided ribbons whose use facilitates computations in a pre-modular category and which underlies applications of this structure to topology and knot theory \cite{TuraevBook} \cite{BKBook}.



The \emph{trace} of an endomorphism $f : A \rightarrow A$ in a pre-modular category is the scalar $\tr(f)$ defined by
\begin{align*}
\ev_{A} \circ c_{A, A^*} \circ (f \otimes \id_{A^*}) \circ \coev_A = \tr(f)\id_{\mathbf{1}},
\end{align*}
and the \emph{quantum dimension} of the object $A$ is by definition
\[
  d(A) = \tr(\id_A).
\]
Suppose the pre-modular category $\CC$ contains $n$ simple objects. Then the un-normalised $S$-matrix of $\CC$ is the $n \times n$ matrix defined by its entries
\[
s_{A,B} = \tr(c_{B,A} \circ c_{A,B}).
\]
If this matrix is invertible then the pre-modular category is said to be modular. The normalised $S$-matrix of $\CC$ is $S = s / D$ where the constant $D$ is $\sqrt{\sum_A d(A)}$. Here the sum extends over the set of simple objects of $\CC$.

The set of fusion rules $\{\mathcal{N}_{A, B}^C\}$ of a pre-modular category $\CC$ is defined by
\[
A \otimes B \cong \bigoplus_C \mathcal{N}_{A,B}^C \cdot C,
\]
where $A, B \in \CC$ are simple objects and the direct sum runs over the set of simple objects of $\CC$. If $\CC$ is modular then the Verlinde formula
\[
\mathcal{N}_{A, B}^C = \sum_{X} \frac{S_{AX} S_{BX} \ov{S_{CX}}}{S_{VX}}
\]
expresses the fusion rules in terms of the $S$-matrix of $\CC$. Here again the sum runs over the set of simple objects of $\CC$, and $\ov{x}$ denotes the complex conjugate of $x$.
\begin{prop}[Huang \cite{H08rigidity}]\label{Huang.modular.cat.prop}
  Let $V$ be a tame vertex algebra. Then the category of $V$-modules possesses the structure of a modular category in which
  \[
    \Hom(A \otimes B, C) \cong \binom{C}{A \,\,\, B},
  \]
the twists of simple objects are given by $\theta_M = e^{2\pi i \D(M)}$, and in which the $S$-matrix coincides with Zhu's $S$-matrix $\rho_V(S)$.
\end{prop}
Inspection of Huang's constructions reveals {\cite[Proposition 3.3.15]{Sven.thesis}} a precise relation between $c$ and $\Om$, namely
\[
  c_{M, M} = \Om(M, M) \cdot e^{2\pi i \D_M} \cdot \id_{M \otimes M}.
\]
We remark that, while the Verlinde formula is a simple consequence of the axioms of a modular category, in practice it is often proved without appeal to the modular category structure and is used in turn to establish the existence of the latter.
\begin{prop}[{\cite{DLN12}}, {\cite{DJX13}}]
Let $V$ be a tame and positive vertex algebra, and $M$ an irreducible $V$-module.
\begin{itemize}
\item The categorical quantum dimension $d(M)$ of $M$ coincides with its analytic quantum dimension, defined to be
\[
\qdim(M) = \lim_{t \rightarrow 0^+} \frac{\chi_M(it)}{\chi_V(it)}.
\]

\item If the operation $M \boxtimes (-)$ permutes the set of irreducible $V$-modules then the analytic quantum dimension of $M$ is $\qdim(M) = +1$.
\end{itemize}
In particular if $V$ has grouplike fusion then $d(W) = +1$ for every irreducible $V$-module $W$.
\end{prop}
The results above (especially Proposition \ref{Huang.modular.cat.prop}) are deep and technical, but once in place they permit Proposition \ref{trace.conf} to be deduced immediately from the following general fact about pre-modular categories \cite{CKL}.
\begin{prop}[{\cite[Proposition 2.32]{DGNO}}]
Let $X$ be a simple object of a pre-modular category. Then
\[
d(X) \cdot \theta_X = \tr(c_{X, X}).
\]
\end{prop}

\section{The Schur Indicator}\label{sec.schur.indic}

Let $G$ be a finite group, and $M$ a self-contragredient irreducible $G$-module defined over $\C$. The isomorphism $M \rightarrow M^\vee$ is the same thing as a nondegenerate bilinear form $(\cdot,\cdot) : M \times M \rightarrow \C$. This form must be either symmetric or else alternating, and in the respective cases we define the \emph{Schur indicator} of $M$ to be $+1$ or else $-1$. Elementary character theory shows that
\[
\nu_M = \frac{1}{\#{G}} \sum_{g \in G} \chi_M(g^2).
\]
The notion of Schur indicator makes sense much more generally, for example for Lie algebras, for vertex algebras, and for modular categories.

Let $\g$ be a finite dimensional simple Lie algebra, and $L(\la)$ a finite dimensional irreducible $\g$-module. It is well known that $L(\la)^\vee \cong L(-w^\circ \la)$ where $w^\circ$ is the longest element of the Weyl group of $\g$. A result of Tits \cite{AG07indicator} states that if $L(\la)$ is self-contragredient then its Schur indicator is
\[
  \nu_{L(\la)} = (-1)^{2 \la(\rho^\vee)},
\]
where $\rho^\vee$ is the dual Weyl vector.

The \emph{contragredient} {\cite[Section 5]{B86}} of a vertex algebra module $M$ is by definition the graded dual vector space $M^*$, equipped with the vertex operation
\begin{align}\label{contragredient.def}
\left<Y(u, z)\varphi, m\right> = \left<\varphi, Y^*(u, z)m\right> \quad \text{where} \quad Y^*(u, z) = Y(e^{zL_1}(-z^{-2})^{L_0}u, z^{-1}).
\end{align}
This rather odd looking definition is natural from the algebro-geometric point of view on vertex algebras in which, in a certain precise sense, the modules $M$ and $M^\vee$ ought to be viewed as being situated at the points $0$ and $\infty$ of the Riemann sphere \cite[Chapter 10]{FBZ}. The contragredient of a $g$-twisted module is defined similarly and is a $g^{-1}$-twisted module. The Schur indicator of a self-contragredient module is then defined as before.

The following proposition is proved by an elementary vertex algebra calculation.
\begin{prop}\label{FSvsSymmetry}
  Let $V$ be a tame vertex algebra, and $M$ a self-contragredient irreducible $V$-module. Suppose the conformal weight $\D(M)$ is a half integer. Then
  \[
    \Om(M, M) = \nu_M \cdot e^{2\pi i \D(M)}.
    \]
\end{prop}
The case $\D(M) \in \Z$ is {\cite[Proposition 5.6.1]{FHL}}, while the case $\D(M) \in \frac{1}{2}+\Z$ is {\cite[Lemma 2.1]{Y12}}.

Consider a tame vertex algebra $V$ with grouplike fusion and fusion group $E$, and let $M$ be an irreducible $V$-module as in Proposition \ref{FSvsSymmetry}. Then $V$ and $M$ determine a subgroup $I$ of $E$ isomorphic to $\Z/2$ and the restriction of the canonical class $\om_V \in H^3_{\text{ab}}(E, \C^\times)$ to $I$ is trivial. Consequently the direct sum $V \oplus M$ carries the structure of a vertex algebra extending that of $V$.

It is straightforward to extend the reasoning above and obtain a proof of Theorem \ref{orb.theorem} for automorphisms of odd order or twice odd order. To prove the theorem in all cases it seems to be necessary to employ the much deeper results of Section \ref{sec.mod.cat}.

\section{Schur Indicator for Lattice Involutions}\label{indic.lattice}

Let $A$ be a group. We recall that central extensions
\begin{align}\label{cent.ext}
1 \rightarrow \C^\times \rightarrow \HH \rightarrow A \rightarrow 1,
\end{align}
of $A$ by $\C^\times$ are classified up to isomorphism by the group cohomology $H^2(A, \C^\times)$ with coefficients in the trivial module $\C^\times$. If $A$ is a finite abelian group then \cite{Yamazaki} the homomorphism
\begin{align*}
  H^2(A, \C^\times) &\rightarrow {\textstyle \bigwedge^2} \Hom_\Z({A}, \C) \\
  [\psi] &\mapsto B, \quad \text{defined by} \quad B(\al, \beta) = \frac{\psi(\al,\beta)}{\psi(\beta,\al)},
\end{align*}
is an isomorphism.

Now let $L$ be a lattice, which we assume to be even and self dual for convenience, and let $\sigma$ be an automorphism of $L$ of order $n$. We consider the sublattices
\[
M = (1-\sigma)L \quad \text{and} \quad N = L \cap (M \otimes_\Z \Q),
\]
both of which are orthogonal to the fixed point sublattice $L^\sigma \subset L$.

Let $A$ be the finite abelian group $N/M$, and let $\HH$ be the central extension (\ref{cent.ext}) associated as above with the alternating form 
\[
B(\al, \beta) = (-1)^{(\al,\beta)}.
\]
If $L$ is self dual then $\HH$ possesses a unique irreducible module $X$, whose construction we now describe. Let $I \subset A$ be a maximal isotropic subspace with respect to the form $B$. Then we have an exact sequence
\[
0 \rightarrow I \rightarrow A \rightarrow \widehat{I} \rightarrow 0,
\]
where $\widehat{I}$ is the Pontryagin dual of $I$ and the projection here sends $\al \in A$ to the character $\chi_\al : i \mapsto B(\al, i)$.

We choose a section $s : \widehat{I} \rightarrow A$ of the projection, and put $X = \C[I]$ with an action of $\HH$ defined by
\[
  e^i \cdot x = {i+x} \quad \text{and} \quad e^{s(\chi)} \cdot x = \chi(x) x.
\]

Let us consider the vertex algebra $V_L$ now, and denote the standard lift of $\sigma$ to $V_L$ also by $\sigma$. The twisted module $V_L(\sigma)$ is constructed as a Fock space over a twisted Heisenberg algebra \cite{dong.lep.96} \cite{BK06}. The lowest graded piece of $V_L(\sigma)$ is a copy of $X$ in which the action of $\left|\al\right>_{(0)}$ for $\al \in N$ coincides with the action of $e^\al \in \HH$.

Now we consider an involution $\sigma$ of $L$ that satisfies
\begin{align}\label{no.order.doubling}
(\al, \sigma \al) \in 2\Z \quad \text{for all $\al \in L$}.
\end{align}
This condition guarantees that the standard lift of $\sigma$ is also an involution. Formula (\ref{lattice.conformal.weight}) implies that the standard lift of $\sigma$ is of type $2\{0\}$ only if the rank of $L^\sigma$ is divisible by $8$. So we assume this from now on.

Now let $(\cdot,\cdot)$ be a nondegenerate $V_L$-invariant bilinear form on $V_L(\sigma)$. We consider its restriction to the lowest graded piece, which we have identified with the $\HH$-module $X$. Knowing that the conformal weight of $\left|\al\right> \in V_L$ is $\al^2/2$, it is straightforward to deduce from (\ref{contragredient.def}) that 
\begin{align}\label{invariance.lowest.piece}
  (e^\al x, y) = (-1)^{\al^2/2}(x, e^\al y)
\end{align}
for all $\al \in A$ and all $x, y \in X$. (See {\cite[equation (4.30)]{BK06}}.)

A simple but key observation is contained in the following lemma.
\begin{lemma}
  The rescaled lattice $M(1/2)$ is even, and its discriminant form is isomorphic to $A$ equipped with the quadratic form $q$ defined by
  \[
q(\al) = \al^2/4 \pmod{1}.
\]
\end{lemma}

\begin{proof}
  That $M(1/2)$ is an even lattice follows from condition (\ref{no.order.doubling}). The rest will follow once we show that $N = 2M^\vee$. Now for any $m \in M$ and $n \in N$ we may write $m = x - \sigma x$ for some $x \in L$, and obtain $(n, m) = (n-\sigma n, x) = 2(n, x) \in 2\Z$. For any $y \in M^\vee$ we have $\Z \ni (y, x-\sigma x) = 2(y, x)$ for all $x \in L$. Hence $y \in \frac{1}{2}L$ and so $2y \in N$.
\end{proof}

To construct the $\HH$-module $X$ we used a maximal isotropic subgroup $I \subset A$. If $I$ happens to be totally isotropic with respect to $q$ then we easily deduce that $(\cdot, \cdot)$ is symmetric. Indeed let $x, y \in I \subset \C[I] = X$ and put $\al = x-y$. Then by assumption $4|\al^2$, and so
\[
(x, y) = (e^\al y, y) = (y, e^\al y) = (y, x).
\]
A discriminant form which possesses a totally isotropic maximal isotropic subgroup is called \emph{split}. So we should like to show that $(A, q)$ is split.

Any discriminant form can be decomposed into a finite direct sum of indecomposable Jordan components {\cite[Section 2]{Scheithauer}}. Since every element of $A$ is of order $2$ and $q(\al) \in \{0, 1/2\}$ the only components that can appear in the decomposition of $A$ are those of type $2_\text{II}^{+2}$ and $2_\text{II}^{-2}$. We now recall the definitions of these two types. As groups $2_\text{II}^{\pm 2} = \Z/2 \times \Z/2$, while the quadratic forms are as shown below.
\begin{center}
\begin{tikzpicture}

\tikzset{
quangleprod/.style={
    matrix of nodes,
    column sep=-\pgflinewidth, row sep=-\pgflinewidth,
    nodes={draw,
      minimum height=#1,
      anchor=center,
      text width=#1,
      align=center, 
      inner sep=0pt 
    },
  },
  quangleprod/.default=1cm
}

\matrix[quangleprod]
{
$0$ & $1/2$ \\
$1/2$ & $1/2$ \\
};

\node at (0, -1.6) {$2_{\text{II}}^{-2}$};

\matrix[quangleprod] at (6,0)
{
$0$ & $0$ \\
$0$ & $1/2$ \\
};

\node at (6, -1.6) {$2_{\text{II}}^{+2}$};
\end{tikzpicture}
\end{center}
The decomposition of a discriminant form is not unique in general, indeed
\begin{align}\label{2.component.switch}
\left( 2_\text{II}^{+2} \right)^{2} \cong \left( 2_\text{II}^{-2} \right)^{2}.
\end{align}






Now we recall the invariant $\ga_2$ of a discriminant form defined by
\begin{align}\label{signature.relation}
\ga_2(D) = e^{2\pi i \sign(L) / 8}
\end{align}
where $D$ is the discriminant form of the lattice $L$, and $\sign(L)$ is the signature of $L$. This invariant is multiplicative, i.e., $\ga_2(D_1 \oplus D_2) = \ga_2(D_1)\ga_2(D_2)$, and satisfies $\ga_2(2_\text{II}^{\pm 2}) = \pm 1$.

We return now to $A = N / M$ the discriminant form of $M(1/2)$. Since $M$ is positive definite and its rank is divisible by $8$, we have $\ga_2(A) = +1$. It follows that the number of copies of $2_\text{II}^{-2}$ in the decomposition of $A$ must be even. But then the relation (\ref{2.component.switch}) implies that $A$ decomposes into a sum of some number of copies of $2_\text{II}^{+2}$. Since $2_\text{II}^{+2}$ is evidently split, so is $A$.

We summarise this discussion in the following theorem.
\begin{thm}\label{QuotientIsSplit}
  Let $L$ be an even self dual lattice, and $\sigma$ an involution of $L$ such that the rank of the fixed point sublattice $L^\sigma$ is divisible by $8$, and such that $(\al, \sigma\al) \in 2\Z$ for all $\al \in L$. Let $\HH$ be the finite Heisenberg group associated with $L$ and $\sigma$ as above, and $X$ the unique irreducible $\HH$-module. Then the unique up to scalar nondegenerate invariant bilinear form on $X$ is symmetric.
\end{thm}
It follows immediately that if the lattice $L$ and its involution $\sigma$ satisfy the conditions of Theorem \ref{QuotientIsSplit} then the Schur indicator of the unique $\sigma$-twisted $V_L$-module $V_L(\sigma)$ is $\nu = +1$.

Now let 
\[
V_L(\sigma) = V_L(\sigma)^+ \oplus V_L(\sigma)^-
\]
denote the decomposition of $V_L(\sigma)$ into its $\Z$-graded and $(\tfrac{1}{2}+\Z)$-graded parts, respectively. The two pieces are irreducible $V_L^\sigma$-modules with Schur indicator $\nu = +1$. Proposition \ref{FSvsSymmetry} implies that the braiding coefficient $\Om$ of $V_L(\sigma)^\pm$ is $\pm 1$. Thus the direct sum $V_L^\sigma \oplus V_L(\sigma)^+$ acquires the structure of a vertex algebra. By the same token the direct sum $V_L^\sigma \oplus V_L(\sigma)^-$ acquires the structure of a vertex superalgebra.

\newpage

\bibliographystyle{plain}
\bibliography{refs-orb}

\end{document}